\documentclass{amsart}
\usepackage{amscd,amssymb,graphicx}

\newcommand{\co}{\colon}
\newcommand{\bR}{\mathbf{R}}
\newcommand{\cE}{\mathcal{E}}
\newcommand{\cR}{\mathcal{R}}
\newcommand{\cT}{\mathcal{T}}
\newcommand{\cW}{\mathcal{W}}
\newcommand{\za}{\alpha}
\newcommand{\ze}{\epsilon}
\newcommand{\zm}{\mu}
\newcommand{\zp}{\pi}
\newcommand{\zr}{\rho}
\newcommand{\zs}{\sigma}
\newcommand{\zt}{\tau}
\newcommand{\zF}{\Phi}
\newcommand{\zV}{\varPhi}

\newtheorem{thm}{Theorem}
\newtheorem{cor}[thm]{Corollary}
\newtheorem{lemma}[thm]{Lemma}

\theoremstyle{definition}
\newtheorem{ex}[thm]{Example}

\newcommand{\sections}{\renewcommand{\thethm}{\thesection.\arabic{thm}}
           \setcounter{thm}{0}}

\newcommand{\nosubsections}{\renewcommand{\thethm}{\thesection.\arabic{thm}}
           \setcounter{thm}{0}}
\newcommand{\linnum}{\stepcounter{thm}\tag{\thethm}}

\begin{document}
\title{Squaring rectangles for dumbbells}

\author{J. W. Cannon}
\address{Department of Mathematics\\ Brigham Young University\\ Provo, UT
84602\\ U.S.A.} \email{cannon@math.byu.edu}

\author{W. J. Floyd}
\address{Department of Mathematics\\ Virginia Tech\\
Blacksburg, VA 24061\\ U.S.A.} \email{floyd@math.vt.edu}
\urladdr{http://www.math.vt.edu/people/floyd}

\author{W. R. Parry}
\address{Department of Mathematics\\ Eastern Michigan University\\
Ypsilanti, MI 48197\\ U.S.A.} \email{walter.parry@emich.edu}

\date{\today}

\begin{abstract}
The theorem on squaring a rectangle (see Schramm \cite{S} and
Cannon-Floyd-Parry \cite{Magnus}) gives a combinatorial version of
the Riemann mapping theorem. We elucidate by example (the dumbbell)
some of the limitations of rectangle-squaring as an approximation to
the classical Riemamnn mapping.
\end{abstract}
\keywords{finite subdivision rule, combinatorial moduli, squaring rectangles}
\subjclass[2000]{Primary 52C20, 52C26}
\maketitle

\sections

\section{Introduction }\label{sec:intro}\nosubsections

A \emph{quadrilateral} is a planar disk $D$ with four distinguished
boundary points $a$, $b$, $c$, and $d$ that appear in clockwise
order on the boundary. These four points determine a \emph{top edge}
$ab$, a \emph{right edge} $bc$, a \emph{bottom edge} $cd$, and a
\emph{left edge} $da$. A \emph{tiling} of $D$ is a finite collection
of disks $T_i$, called tiles, whose union fills $D$, whose interiors
\emph{Int} $T_i$ are disjoint, and whose boundaries form a finite
graph $\Gamma$ in $D$ that contains $\partial D$.

The \emph{rectangle-squaring theorem} (see Schramm \cite{S} and
Cannon-Floyd-Parry \cite{Magnus}) states that there are integers
$p_i \ge 0$ parametrized by the tiles (for a tile $T_i$, $p_i$ is
called the \emph{weight} of $T_i$), not all equal to $0$, unique up
to scaling, such that squares $S_i$ of edge length $p_i$ can be
assembled with essentially the same adjacencies as the tiles $T_i$
to form a geometric rectangle with top, bottom, and sides
corresponding to the edges of $D$. Some combinatorial distortions
are inevitable. For example, in a squared rectangle, at most four
tiles can meet at a point. The distortions allowed are these: (1) a
vertex may expand into a vertical segment and (2) a tile may
collapse to a point. No other distortions are required.

The rectangle-squaring theorem is essentially a combinatorial
version of the Riemann Mapping Theorem. It has the advantage over
other versions of the Riemann Mapping Theorem that the integers
$p_i$ can be calculated by a terminating algorithm and can be
approximated rapidly by various simple procedures. As a consequence,
rectangle-squaring can be used as a rapid preprocessor for other
computational methods of approximating the Riemann mapping.

Schramm \cite{S} has pointed out that tilings given by the simplest
subdivision rules give squarings that need not converge to the
classical Riemann mapping. The purpose of this paper is to further
elucidate the limitations of the method.  In the classical Riemann
mapping, changing the domain of the mapping anywhere typically
changes the mapping everywhere. Our main result shows that this is
not true for rectangle squaring. A \emph{dumbbell} is a planar
quadrilateral, constructed from squares of equal size, that consists
of two blobs (the left ball and the right ball) at the end joined by
a relatively narrow bar of uniform height in the middle. We show
that for any choices of the left and right balls of a dumbbell, the
weights $p_i$ associated with the squares in the middle of the bar
are constant provided that the bar is sufficiently long and narrow.
In particular, subdivision cannot lead to tilings whose squarings
converge to the classical Riemann mapping.

In order to state our main theorem, we precisely define what we mean
by a dumbbell. A \emph{dumbbell} $D$ is a special sort of
conformal quadrilateral which is a subcomplex of the square tiling
of the plane. It consists of a \emph{left ball}, a \emph{right
ball}, and a \emph{bar}. The left ball, the bar, and the right
ball are all subcomplexes of $D$. The bar is a rectangle at least
six times as wide as high. It meets the left ball in a connected
subcomplex of the left side of the bar, and it meets the right ball
in a connected subcomplex of the right side of the bar. The bar of a
dumbbell is never empty, but we allow the balls to be empty. The
\emph{bar height} of $D$ is the number of squares in each column
of squares of the bar of $D$. Figure~\ref{fig:dumb} shows a dumbbell
with bar height $1$.

\begin{figure}[ht]
\centerline{\includegraphics{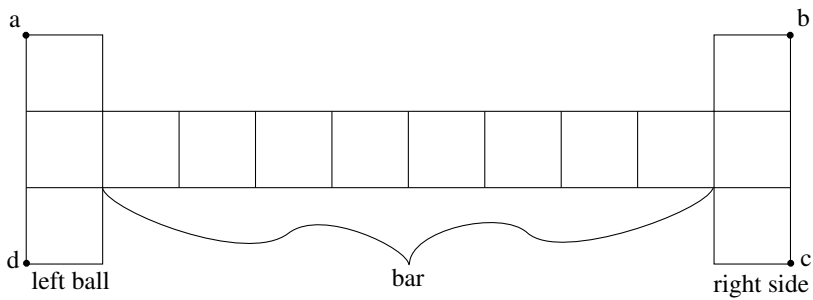}} \caption{A dumbbell $D$}
\label{fig:dumb}
\end{figure}

If $D$ is a dumbbell with bar height $n$, then a weight function
$\zr$ for $D$ is \emph{virtually bar uniform} if
$\zr(t)=\frac{1}{n}H_\zr$ for any tile $t$  in the bar of $D$ whose
skinny path distance to the balls of $D$ is at least $3n$. (The
skinny path distance between two subsets of $D$ is one less than the
minimum length of a chain of tiles which joins a point in one subset
to a point in the other subset. The skinny path distance is used
extensively in \cite{expi}.) We are now ready to state our main
theorem.

\medskip\noindent\textbf{Dumbbell Theorem.} Every fat
flow optimal weight function for a dumbbell is virtually bar
uniform.
\smallskip

As indicated above, the theorem has consequences for Riemann
mappings. Let $D$ be a dumbbell with bar height $n$, and let $\zr$
be a fat flow optimal weight function for $D$. If $t$ is a tile in
the bar of $D$ whose skinny path distance to the balls of $D$ is at
least $3n$, then the dumbbell theorem implies that
$\zr(t)=\frac{1}{n}H_\zr$. Now we subdivide $D$ using the binary
square finite subdivision rule $\cR$, which subdivides each square
into four subsquares. We see that $\cR(D)$ is a dumbbell with bar
height $2n$.  Let $s$ be a tile of $\cR(D)$ contained in $t$. Then
the skinny path distance from $s$ to the balls of $\cR(D)$ is at
least $6n$.  If $\zs$ is a fat flow optimal weight function for
$\cR(D)$, then the dumbbell theorem implies that
$\zs(s)=\frac{1}{2n}H_\zs$. It follows that as we repeatedly
subdivide $D$ using the binary square finite subdivision rule and
normalize the optimal weight functions so that they have the same
height, if they converge, then they converge to the weight function
of an affine function in the middle of the bar of $D$.  The only way
that a Riemann mapping of $D$ can be affine in the middle of the bar
of $D$ is for $D$ to be just a rectangle. Thus our sequence of
optimal weight functions almost never converges to the weight
function of a Riemann mapping of $D$. We formally state this result
as a corollary to the dumbbell theorem.

\bigskip\noindent\textbf{Corollary.} When we repeatedly subdivide a
dumbbell using the binary square subdivision rule, if the resulting
sequence of fat flow optimal weight functions converges to the
weight function of a Riemann mapping, then the dumbbell is just a
rectangle.
\bigskip

Here is a brief outline of the proof of the dumbbell theorem.  Let
$D$ be a dumbbell, and let $\zr$ be a fat flow optimal weight
function for $D$. From $\zr$ we construct a new weight function
$\zs$ for $D$. The weight function $\zs$ is constructed so that it
is virtually bar uniform and if $t$ is a tile of $D$ not in the bar
of $D$, then $\zs(t)=\zr(t)$.  Much effort shows that $H_\zs=H_\zr$.
Because $\zs$ is a weight function for $D$ with $H_\zs=H_\zr$ and
$\zr$ is optimal, $A_\zr\le A_\zs$. Because $\zs$ agrees with $\zr$
outside the bar of $D$, it follows that there exists a column of
squares in the bar of $D$ whose $\zr$-area is at most its
$\zs$-area.  The main difficulty in the proof lies in controlling
the $\zs$-areas of such columns of squares.  The key result in this
regard is Theorem~\ref{thm:estimate}. Most of this proof is devoted
to proving Theorem~\ref{thm:estimate}. It gives us enough
information about the $\zs$-areas of such columns of squares to
conclude that $H_\zs=H_\zr$ and eventually that $\zs=\zr$. Thus
$\zr$ is virtually bar uniform.

We give some simple examples in Section~\ref{sec:examples}.
Sections~\ref{sec:vectors} and \ref{sec:cuts} are relatively easy.
The skinny cut function $\zV$ is defined in the first paragraph of
Section~\ref{sec:function}.  Once the reader understands the
definition of $\zV$, the statement of Theorem~\ref{thm:estimate} can
be understood.  After understanding the statement of
Theorem~\ref{thm:estimate}, the reader can read
Section~\ref{sec:proof} to get a better grasp of the proof of the
dumbbell theorem outlined in the previous paragraph.
Theorem~\ref{thm:estimate} is the key ingredient, and its proof
presents the greatest difficulties in our argument.

In Section~\ref{sec:notes} we discuss without proofs an assortment
of results which are related to (but not used in) our proof of the
dumbbell theorem.

\section{Examples}\label{sec:examples}\nosubsections

We give some simple examples here to illustrate the theorem.

\begin{ex}
We begin with an example that motivated this work. Consider the
topological disk $D_1$ shown in Figure~\ref{fig:ell0}. The left ball
of $D_1$ is a union of two tiles, the bar of $D_1$ is a single tile,
and the right ball of $D_1$ is empty. $D_1$ is not a dumbbell
because the bar isn't wide enough, but since it has so few tiles it
is easier of analyze. Figure~\ref{fig:ellsubs} shows the first three
subdivisions of $D_1$ with respect to the binary square subdivision
rule; the tiling $\cR^{n+1}(D_1)$ is obtained from $\cR^n(D_1)$ by
subdividing each square into four subsquares.  We consider each
subdivision as a conformal quadrilateral by choosing the same four
points as vertices and the same labeling of the edges.
\end{ex}

\begin{figure}[ht]
\centerline{\includegraphics{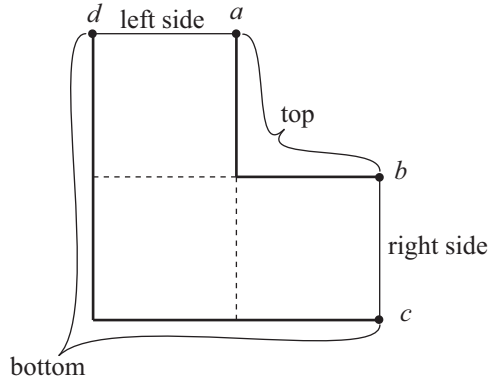}} \caption{The quadrilateral
$D_1$} \label{fig:ell0}
\end{figure}

\begin{figure}[ht]
\centerline{\includegraphics{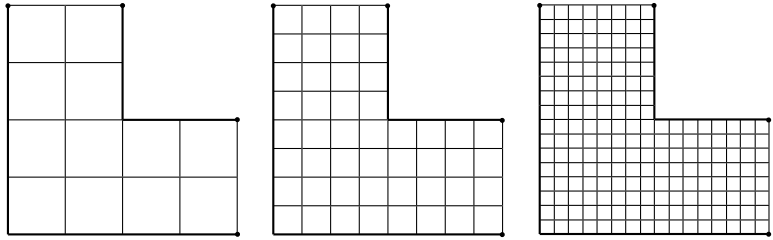}} \caption{The first three
subdivisions of $D_1$} \label{fig:ellsubs}
\end{figure}

We considered this example because we were interested in the
squared rectangles corresponding to the optimal weight functions
for the fat flow moduli of this sequence of quadrilaterals. Given
a tiling $\cT$ (or, more generally, a shingling) of a conformal
quadrilateral $Q$, a {\em weight function} on $\cT$ is a
non-negative real-valued function $\rho$ on the set of tiles of
$\cT$. If $\rho$ is a weight function on $\cT$ and $t$ is a tile
of $\cT$, then $\rho(t)$ is the {\em weight} of $t$. One can use a
weight function to assign ``lengths'' to paths in $Q$ and
``areas'' to subsets of $Q$. The {\em $\rho$-length} of a path
$\alpha$ in $Q$ is the sum of the weights of the tiles that the
image of $\alpha$ intersects, and the {\em $\rho$-area}
$A_\rho(W)$ of a subset $W$ of $Q$ is the sum of the squares of
the weights of the tiles that intersect $W$. One then defines the
{\em $\rho$-height} $H_\rho$ of $Q$ to be the minimum
$\rho$-length of a path in $Q$ that joins the top and bottom of
$Q$, and one defines the {\em $\rho$-area} $A_{\rho}$ of $Q$ to be
$A_{\rho}(Q)$. The {\em fat flow modulus} of $Q$ with respect to
$\rho$ is $H_{\rho}^2/A_{\rho}$. The {\em optimal weight function}
is a weight function whose fat flow modulus is the supremum of the
fat flow moduli of weight functions; it exists and is unique up to
scaling. There is a squared rectangle corresponding to $Q$ whose
squares correspond to the tiles of $\cT$ with non-zero weights
under the optimal weight function; furthermore, the side length of
one of these squares is the weight of the corresponding tile.

\begin{figure}[ht]
\centerline{\includegraphics{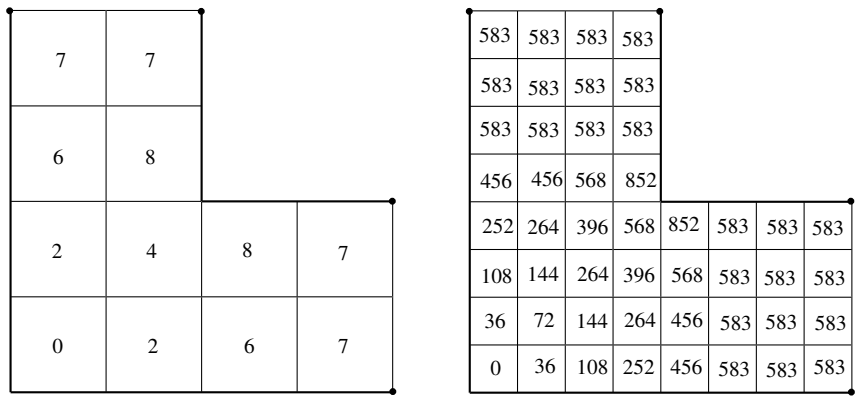}} \caption{Optimal weight
functions for the first two subdivisions of $D_1$}
\label{fig:ellwts}
\end{figure}

\begin{figure}[ht]
\centerline{\includegraphics{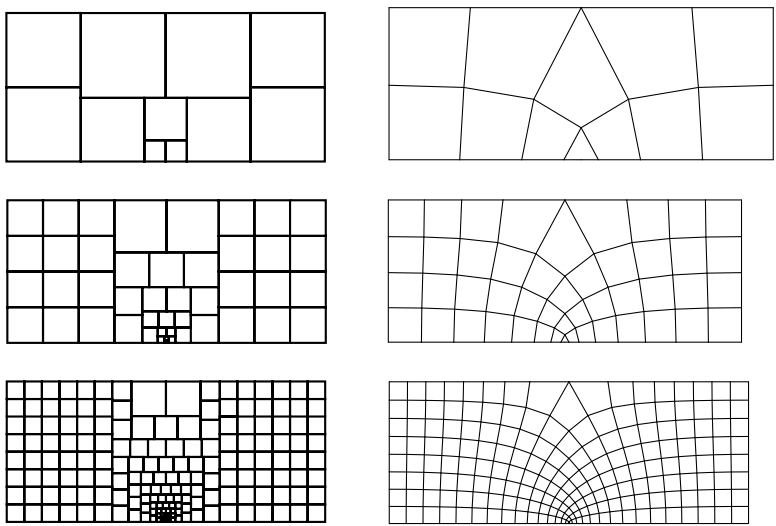}} \caption{The squared
rectangles and circle packings for the first three subdivisions of
$D_1$} \label{fig:ellsrcp}
\end{figure}

Figure~\ref{fig:ellwts} gives optimal weight functions for the
first two subdivisions of the quadrilateral $D_1$, and the left
side of Figure~\ref{fig:ellsrcp} gives the squared rectangles for
the first three subdivisions of $D_1$. (See \cite{Magnus} for
detailed information about optimal weight functions for tilings.)
For each of these tilings of $D_1$, one can define a circle
packing whose carrier complex is obtained from the tiling by
adding a barycenter to each tile and then subdividing each tile
into triangles by adding an edge from each vertex of the tile to
the barycenter. By only drawing the edges of the packed carrier
complex that correspond to edges of the tiling, one can use the
circle packing to draw the tilings. On the right side of
Figure~\ref{fig:ellsrcp}, this is done for the first three
subdivisions of $D_1$ using Stephenson's program CirclePack
\cite{CP}.

While one cannot expect discrete functions from squared rectangles
for tilings to converge to Riemann maps, He and Schramm showed in
\cite{HS} that under general hypotheses discrete functions
associated to circle packings converge to Riemann maps. It may seem
that the squared rectangles and the circle packings are not closely
related, but in some computational examples with the pentagonal and
dodecahedral subdivision rules we have found that using our squared
rectangle software \cite{precp} in conjuction with CirclePack
\cite{CP} can lead to dramatic reductions in the total computation
time for producing the packings. (In an example with over 1,600,000
vertices, using both programs led to a reduction in the computation
time from almost 38 hours to under six hours.) We wondered whether
the squared rectangles for the subdivisions of $D_1$ would define
discrete functions which were closely related to the Riemann map for
$D_1$.

But in this example the optimal weight functions appeared to be
constant near the two sides of the quadrilateral. As we described
above, if they really are constant then they couldn't possibly give
discrete approximations to the Riemann map. We found this
``apparent'' constancy of weights near the sides to be very
surprising, and decided to look more closely at optimal weight
functions for quadrilaterals made out of square tiles.

\begin{ex}
Let $D$ be the dumbbell drawn in Figure~\ref{fig:dumb}. The left
edge of $D$ is the left side of the left ball, and the right edge
of $D$ is the right side of the ride ball. The bar of $D$ has
height $1$ and width $8$. Figure~\ref{fig:dumbl1-3} gives squared
rectangles for the first three subdivisions of $D$. For these
subdivisions, the weights $p_i$ are constant on the entire bar and
not just on the middle fourth of the bar.
\end{ex}

\begin{figure}[ht]
\centerline{\includegraphics{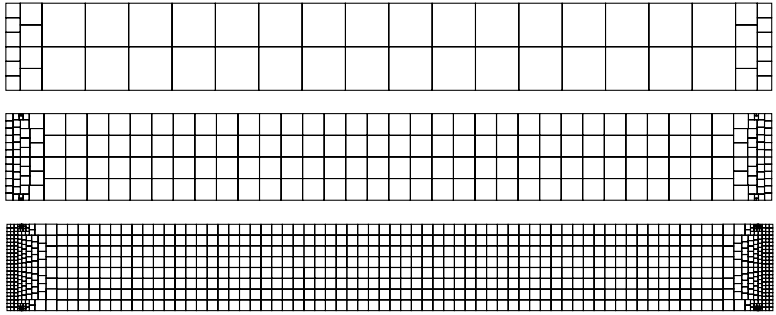}} \caption{Squared
rectangles for the first three subdivisions of $D$}
\label{fig:dumbl1-3}
\end{figure}

\begin{ex}
The dumbbell $D_2$ shown in Figure~\ref{fig:tray} is similar to the
dumbbell $D$. The two dumbbells have the same bars, but for $D_2$
the left and right balls are smaller. The left edge of $D_2$ is the
top side of the left ball and the right edge of $D_2$ is the top
side of the right ball. Squared rectangles for the first three
subdivisions of $D_2$ are shown in Figure~\ref{fig:trayl1-3}. The
dumbbell theorem guarantees that the weights are constant in the
middle fourth of the bar, but here they are constant on most of the
bar. The quadrilateral $D_1$ is a subcomplex of $D_2$.
\end{ex}

\begin{figure}[ht]
\centerline{\includegraphics{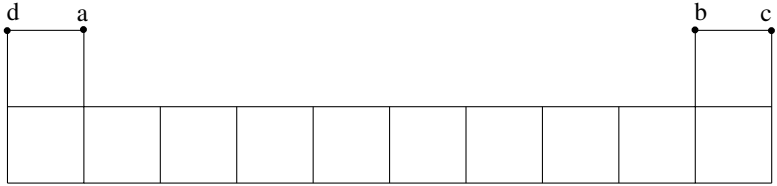}} \caption{The dumbbell $D_2$}
\label{fig:tray}
\end{figure}

\begin{figure}[ht]
\centerline{\includegraphics{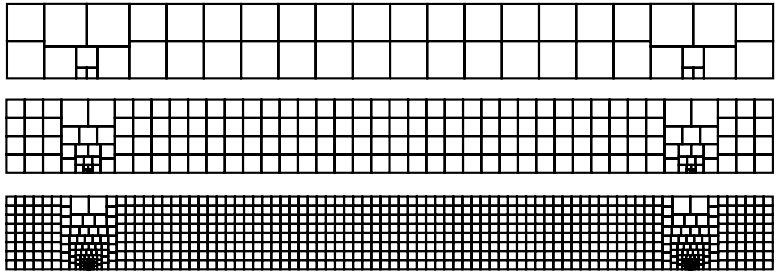}} \caption{Squared
rectangles for the first three subdivisions of $D_2$}
\label{fig:trayl1-3}
\end{figure}

\section{Weight vectors }\label{sec:vectors}\nosubsections

This section deals with basic properties of vectors in Euclidean
space.

We fix a positive integer $n$, which will be the bar height of a
dumbbell under consideration. A \emph{weight vector} is an element
of $\bR^n$ whose components are all nonnegative and not all 0. We
denote by $\cW$ the set of all weight vectors in $\bR^n$. The
\emph{height} $H(x)$ and \emph{area} $A(x)$ of a weight vector
$x=(x_1,\dotsc,x_n)$ are defined by
  \begin{equation*}
H(x)=\sum_{i=1}^{n}x_i \quad\text{and}\quad
A(x)=\sum_{i=1}^{n}x_i^2.
  \end{equation*}
Given a positive real number $h$, we let
$w_h={\textstyle(\frac{h}{n},\dotsc,\frac{h}{n})}$ and
  \begin{equation*}
\cW_h=\{x\in \cW:H(x)=h\}.
  \end{equation*}
We see that $w_h\in \cW_h$ and that $A(w_h)=\frac{h^2}{n}$.

\begin{lemma}\label{lemma:distance}  Let $h$ be a positive
real number, let $x_1,x_2\in \cW_h$, and let $d_1$ and $d_2$ be the
distances from $x_1$ and $x_2$ to $w_h$.  Then $A(x_1)\le A(x_2)$ if
and only if $d_1\le d_2$.
\end{lemma}
  \begin{proof} Let $x\in \cW_h$.  Since $H(x)$ is the dot product of
$x$ with the vector whose components are all 1, we have that
$x\cdot(\frac{n}{h}w_h)=h$.  Hence $x\cdot w_h=\frac{h^2}{n}=w_h\cdot
w_h$.  So the square of the distance from $x$ to $w_h$ is
  \begin{equation*}
(x-w_h)\cdot(x-w_h)=x\cdot x-2x\cdot w_h+w_h\cdot w_h=x\cdot
x-w_h\cdot w_h.
  \end{equation*}
Thus decreasing $A(x)$ is equivalent to decreasing the distance from
$x$ to $w_h$.  This proves the lemma.

\end{proof}

\begin{cor}\label{cor:wh}  The weight vector $w_h$ is the unique
element of $\cW_h$ with least area.
\end{cor}

\begin{lemma}\label{lemma:decreasing}  Let $x$ and $y$ be
distinct weight vectors such that $A(x)\ge A(y)$.  Then the area
function restricted to the line segment from $x$ to $y$ is strictly
decreasing at $x$.
\end{lemma}
  \begin{proof} The line segment from $x$ to $y$ is traversed by
$(y-x)t+x$ as the parameter  $t$ varies from 0 to 1.  We have that
  \begin{equation*}
A((y-x)t+x)=((y-x)t+x)\cdot((y-x)t+x)=(y-x)\cdot(y-x)t^2+2x\cdot(y-x)t+x\cdot
x.
  \end{equation*}
Viewing this as a function of $t$, its derivative at $t=0$ is
$2x\cdot(y-x)$.  But since
  \begin{equation*}
x\cdot x-2x\cdot y+y\cdot y=(x-y)\cdot(x-y)>0,
  \end{equation*}
we have that $2x\cdot y<x\cdot x+y\cdot y$.  So
  \begin{equation*}
2x\cdot(y-x)=2x\cdot y-2x\cdot x<y\cdot y-x\cdot x=A(y)-A(x)\le0.
  \end{equation*}
This proves the lemma.

\end{proof}

\begin{lemma}\label{lemma:minimum}  Every nonempty compact
convex set of weight vectors in $\bR^n$ contains a unique weight
vector with minimal area.
\end{lemma}
  \begin{proof} Let $C$ be a nonempty compact convex set of weight
vectors in $\bR^n$.  Since the area function is continuous, it has a
minimum on $C$.  Let $x$ and $y$ be weight vectors in $C$ with minimal
area.  Lemma~\ref{lemma:decreasing} implies that if $x\ne y$, then the
area function restricted to the line segment from $x$ to $y$ is
decreasing at $x$.  This is impossible because $x$ and $y$ have
minimal area, and so $x=y$.

This proves Lemma~\ref{lemma:minimum}.

\end{proof}

\section{Weight functions which are sums of strictly monotonic cuts
}\label{sec:cuts}\nosubsections

This section deals with weight functions on rectangles that are
viewed as conformal quadrilaterals.  Eventually such a rectangle
will be chosen to be contained in the bar of a dumbbell with the top
of the rectangle in the top of the bar and the bottom of the
rectangle in the bottom of the bar.

Let $R$ be a rectangle in the plane tiled in the straightforward
way by squares with $n$ rows and $m$ columns of squares.  Let
$T_{ij}$ be the square in row $i$ and column $j$ for
$i\in\{1,\ldots,n\}$ and $j\in\{1,\ldots,m\}$.  With an eye toward
combinatorial conformal moduli, we view $R$ as a quadrilateral in
the straightforward way. A skinny cut for $R$ is \emph{strictly
monotonic} if it contains exactly one tile in every column of $R$.
A weight function $\zr$ on $R$ is a \emph{sum of strictly
monotonic skinny cuts} if $\zr$ is a nonnegative linear
combination of characteristic functions of strictly monotonic
skinny cuts.

\begin{lemma}\label{lemma:inequalities}  In the
situation of the previous paragraph, the weight function $\zr$ is a
sum of strictly monotonic skinny cuts if and only if for every
$j\in\{1,\ldots,m-1\}$ we have that
  \begin{equation*}
\sum_{i=1}^{n}\zr(T_{ij})=\sum_{i=1}^{n}\zr(T_{i\,j+1})
  \end{equation*}
and
  \begin{equation*}
\sum_{i=1}^{k}\zr(T_{ij})\le\sum_{i=1}^{k+1}\zr(T_{i\,j+1})\text{
and }
\sum_{i=1}^{k}\zr(T_{i\,j+1})\le\sum_{i=1}^{k+1}\zr(T_{ij})\text{
for every }k\in\{1,\dotsc,n-1\}.
  \end{equation*}
\end{lemma}
  \begin{proof} Suppose that $\zr$ is a sum of strictly monotonic
skinny cuts.  Let $j\in\{1,\ldots,m-1\}$, and let
$k\in\{1,\ldots,n-1\}$.  Then every skinny cut defining $\zr$ which
contains one of $T_{1j},\dotsc,T_{kj}$ also contains one of
$T_{1\,j+1},\dotsc,T_{k+1\,j+1}$.  Since $\zr$ is a nonnegative
linear combination of such skinny cuts, it follows that
  \begin{equation*}
\sum_{i=1}^{k}\zr(T_{ij})\le\sum_{i=1}^{k+1}\zr(T_{i\,j+1}).
  \end{equation*}
Likewise
  \begin{equation*}
\sum_{i=1}^{k}\zr(T_{i\,j+1})\le\sum_{i=1}^{k+1}\zr(T_{ij}).
  \end{equation*}
In the case $k=n$, every skinny cut defining $\zr$ which contains
one of $T_{1j},\dotsc,T_{nj}$ also contains one of
$T_{1\,j+1},\dotsc,T_{n\,j+1}$. So
  \begin{equation*}
\sum_{i=1}^{n}\zr(T_{ij})\le\sum_{i=1}^{n}\zr(T_{i\,j+1}).
  \end{equation*}
The opposite inequality holds by symmetry, and so the inequality is
actually equality.  This proves the forward implication of the lemma.

For the backward implication, suppose that these equalities and
inequalities are satisfied.  We argue by induction on the number $N$
of tiles $T_{ij}$ such that $\zr(T_{ij})\ne0$.  The statement to be
proved is vacuously true if $N=0$.  So suppose that $N>0$ and that
the statement is true for smaller values of $N$.  Let
$j\in\{1,\dotsc,m\}$.  Since the zero function is not an allowable
weight function and since column sums of $\zr$-weights are equal, it
follows that the sum of the $\zr$-weights of the tiles in column $j$
is not 0.  Let $S_j$ be the tile $T_{ij}$ with the smallest value of
$i$ such that $\zr(T_{ij})\ne0$.  The inequalities imply that
$\{S_1,\dotsc,S_m\}$ is a strictly monotonic skinny cut.  Let
$c=\min\{\zr(S_1),\dotsc,\zr(S_m)\}$.  Let $\zr'$ be the function on
the tiles of $R$ gotten from $\zr$ by subtracting $c$ times the
characteristic function of $\{S_1,\dotsc,S_m\}$.  If $\zr'$ is the
zero function, then $\zr$ is just $c$ times the characteristic
function of $\{S_1,\dotsc,S_m\}$.  Otherwise $\zr'$ is a weight
function which satisfies the inequalities and it has value 0 at more
tiles than does $\zr$.  So by induction $\zr'$ is a sum of strictly
monotonic skinny cuts.  It follows that $\zr$ is too.

This completes the proof of Lemma~\ref{lemma:inequalities}.

\end{proof}

\begin{lemma}\label{lemma:exists}  Let $R$ be a rectangle as
above tiled by $n$ rows and $m$ columns of squares $T_{ij}$.  Let
$x=(x_1,\dotsc,x_n)$ be a weight vector in $\bR^n$.  Then there exists
a unique weight function $\zr$ for $R$ with minimal area subject to
the conditions that $\zr$ is a sum of strictly monotonic skinny cuts
and $\zr(T_{i1})=x_i$ for every $i\in\{1,\ldots,n\}$.
\end{lemma}
  \begin{proof} If $\zr$ is a weight function for $R$ which is a sum
of strictly monotonic skinny cuts such that $\zr(T_{i1})=x_i$ for
$i\in\{1,\ldots,n\}$, then the $\zr$-height of $R$ is $H(x)$.  Thus to
minimize area, we may restrict to weight functions whose weights are
all at most $H(x)$.  Lemma~\ref{lemma:inequalities} shows that we are
minimizing area over a compact convex subset of $\bR^{nm}$.  This
subset of $\bR^{nm}$ is nonempty because it contains the weight vector
corresponding to the weight function $\zr$ for $R$ for which
$\zr(T_{ij})=x_i$ for every $i$ and $j$.  Now
Lemma~\ref{lemma:minimum} completes the proof of
Lemma~\ref{lemma:exists}.

\end{proof}

\section{The skinny cut function }\label{sec:function}\nosubsections

We maintain the setting of Section~\ref{sec:cuts},  and continue to
let $\cW$ denote the set of all weight vectors in $\bR^n$.
Lemma~\ref{lemma:exists} allows us to define a function $\zF:\cW\to
\cW$ as follows.  Let $x=(x_1,\dotsc,x_n)\in \cW$.  We apply
Lemma~\ref{lemma:exists} in the case where $m=2$, so that the
rectangle $R$ has $n$ rows and only two columns. Then there exists a
unique weight function $\zr$ for $R$ with minimal area subject to
the conditions that $\zr$ is a sum of strictly monotonic skinny cuts
and $\zr(T_{i1})=x_i$ for $i\in\{1,\ldots,n\}$.  Let
$y=(y_1,\dotsc,y_n)$ be the weight vector such that $y_i =
\zr(T_{i2})$ for $i\in\{1,\ldots,n\}$.  We set $\zF(x)=y$. This
defines $\zF$, which we call the \emph{skinny cut function}. This
section is devoted to the investigation of $\zF$.

Let $x\in \cW$, and let $h=H(x)$.  Lemma~\ref{lemma:inequalities}
implies that $H(\zF(x))=h$, so $\zF$ maps $\cW_h$ into $\cW_h$.
Moreover it is not difficult to see that if $r$ is a positive real
number, then $\zF(rx)=r\zF(x)$.

Let $h$ be a positive real number, and let $x\in \cW_h$.  From $x$
we obtain real numbers $p_0=0\le p_1\le p_2\le\cdots\le p_n=h$ by
setting $p_k=\sum_{i=1}^{k}x_i$ for $k\in\{0,\ldots,n\}$.  We denote
$(p_0,\dotsc,p_n)$ by $\zp(x)$ and we call it a \emph{weak
partition} of the closed interval $[0,h]$ as opposed to a
\emph{strict partition} in which the numbers $p_0,\dotsc,p_n$ are
required to be distinct.  We call $p_k$ the $k$th \emph{partition
point} of $x$.  This correspondence gives a bijection between
$\cW_h$ and the set of all weak partitions of $[0,h]$.

The inequalities of Lemma~\ref{lemma:inequalities} can be easily
interpreted in terms of partition points.  Let $R$ be a rectangle as
before with $n$ rows and two columns of tiles. Let $\zr$ be a weight
function on $R$ which is a sum of strictly monotonic skinny cuts.
Let $x=(x_1,\dotsc,x_n)$ and let $y=(y_1,\dotsc,y_n)$, where
$x_i=\zr(T_{i1})$ and $y_i=\zr(T_{i2})$ for each $i$. Let
$(p_0,\dotsc,p_n)=\zp(x)$, and let $(q_0,\dotsc,q_n)=\zp(y)$.
Lemma~\ref{lemma:inequalities} implies that $H(x)=H(y)$. Then,
assuming that $H(x)=H(y)$, the inequalities of
Lemma~\ref{lemma:inequalities} are equivalent to the inequalities
$q_{k-1}\le p_k\le q_{k+1}$ for every $k\in\{1,\ldots,n-1\}$.
Likewise, assuming that $H(x)=H(y)$, the inequalities of
Lemma~\ref{lemma:inequalities} are equivalent to the inequalities
$p_{k-1}\le q_k\le p_{k+1}$ for every $i\in\{1,\ldots,n-1\}$.  We
say that two weight vectors $x$ and $y$ are \emph{compatible} if
$H(x)=H(y)$ and $q_{k-1}\le p_k\le q_{k+1}$, equivalently
$p_{k-1}\le q_k\le p_{k+1}$, for every $k\in\{1,\ldots,n-1\}$, where
$(p_0,\dotsc,p_n)=\zp(x)$ and $(q_0,\dotsc,q_n)=\zp(y)$.  This gives
us the following reformulation of the definition of $\zF$, formally
stated as a lemma.

\begin{lemma}\label{lemma:reformulate}  If $x$ is a
weight vector, then $\zF(x)$ is the weight vector with minimal area
which is compatible with $x$.
\end{lemma}

Let $x=(x_1,\dotsc,x_n)\in \cW$, and let $(p_0,\dotsc,p_n)=\zp(x)$.
Let $k\in\{1,\ldots,n-1\}$.  We call $p_k$ a \emph{left leaner} for
$x$ if $x_k>x_{k+1}$, and we call $p_k$ a \emph{right leaner} for
$x$ if $x_k<x_{k+1}$.  The rest of this paragraph explains this
terminology.  Suppose that $p_k$ is a left leaner for $x$ for some
$k\in\{1,\ldots,n-1\}$.  Then $x_k>x_{k+1}$.  The function
$f(t)=(t-p_{k-1})^2+(p_{k+1}-t)^2$ is convex on the closed interval
$[p_{k-1},p_{k+1}]$ with a minimum at the interval's midpoint.  This
and the inequality $x_k>x_{k+1}$ imply that decreasing $p_k$ while
fixing all other partition points of $x$ decreases the area of $x$.
So we view $p_k$ as leaning left toward a position of less area for
$x$.  A similar discussion holds for right leaners.

Lemma~\ref{lemma:reformulate} and the previous paragraph imply for
every $x\in \cW$ that $A(\zF(x))<A(x)$ unless $x$ has no leaners.  But
if $x$ has no leaners, then $x=w_h$, where $h=H(x)$.
Corollary~\ref{cor:wh} implies that $\zF(w_h)=w_h$.  Thus for every
positive real number $h$ the restriction of $\zF$ to $\cW_h$ has a
unique fixed point, namely, $w_h$.  We formally state some of the
results just proved in the following lemma.

\begin{lemma}\label{lemma:reducearea}  For every positive
real number $h$ the restriction of $\zF$ to $\cW_h$ is a function
$\zF\big|_{\cW_h}\co \cW_h\to \cW_h$ which reduces areas of weight
vectors except for at the unique fixed point, $w_h$.
\end{lemma}

Now let $x$ and $y$ be compatible weight vectors with
$(p_0,\dotsc,p_n)=\zp(x)$ and $(q_0,\dotsc,q_n)=\zp(y)$.  Suppose that
$q_k$ is a left leaner for $y$ for some $k\in\{1,\ldots,n-1\}$.  We
say that $x$ \emph{blocks} $q_k$ if $q_k=p_{k-1}$.  The motivation
behind this terminology is that for $x$ and $y$ to be compatible the
inequality $q_k\ge p_{k-1}$ must be satisfied, and so even though
$q_k$ is leaning left, $x$ prevents us from decreasing $q_k$
to decrease the area of $y$.  Similarly, if $q_k$ is a right leaner for
$y$, then we say that $x$ blocks $q_k$ if $q_k=p_{k+1}$.

\begin{lemma}\label{lemma:leaner}  Let $x$ and $y$ be
compatible weight vectors.  Then every left leaner for $y$ which is
blocked by $x$ is a left leaner for $x$.  Similarly, every right
leaner for $y$ which is blocked by $x$ is a right leaner for $x$.
\end{lemma}
  \begin{proof} Suppose that $x=(x_1,\dotsc,x_n)$,
$y=(y_1,\dotsc,y_n)$, $\zp(x)=(p_0,\dotsc,p_n)$, and
$\zp(y)=(q_0,\dotsc,q_n)$.  Suppose that $q_k$ is a left leaner for
$y$ blocked by $x$ for some $k\in\{1,\ldots,n-1\}$.  This means that
$y_k>y_{k+1}$ and that $q_k=p_{k-1}$.  Since $q_{k-1}\ge p_{k-2}$,
we have $y_k\le x_{k-1}$.  Since $q_{k+1}\ge p_k$, we have
$y_{k+1}\ge x_k$.  Combining these inequalities, we obtain that
$x_{k-1}\ge y_k>y_{k+1}\ge x_k$, and so $p_{k-1}$ is a left leaner
for $x$. Hence every left leaner for $y$ which is blocked by $x$ is
a left leaner for $x$. A similar argument holds for right leaners.

This proves Lemma~\ref{lemma:leaner}.

\end{proof}

\begin{lemma}\label{lemma:leaner'}  Let $x\in \cW$.  Then
every left leaner for $\zF(x)$ is a left leaner for $x$ blocked by
$x$, and every right leaner for $\zF(x)$ is a right leaner for $x$
blocked by $x$.  Conversely, if $y$ is a weight vector compatible with
$x$ such that every leaner for $y$ is blocked by $x$, then $y=\zF(x)$.
\end{lemma}
  \begin{proof} Let $q$ be a left leaner for $\zF(x)$.  Then $q$ is
blocked by $x$, for otherwise it is possible to decrease $A(\zF(x))$
by decreasing $q$.  Thus every left leaner for $\zF(x)$ is blocked by
$x$.  Now Lemma~\ref{lemma:leaner} implies that every left leaner for
$\zF(x)$ is a left leaner for $x$.  Similarly, every right leaner for
$\zF(x)$ is a right leaner for $x$ blocked by $x$.

Now suppose that $y$ is a weight vector compatible with $x$ such that
every leaner for $y$ is blocked by $x$.  Suppose that $y\ne \zF(x)$.
All weight vectors on the line segment from $y$ to $\zF(x)$ are
compatible with $x$.  Lemma~\ref{lemma:decreasing} implies that the
area function restricted to this line segment is strictly decreasing
at $y$.  But as we move the partition points of $y$ linearly toward
the partition points of $\zF(x)$ we either move leaners away from
blocked positions, which increases area, or we move nonleaners, which
also increases area.  Thus $y=\zF(x)$.

This proves Lemma~\ref{lemma:leaner'}.

\end{proof}

We next introduce the notion of segments.  Let $x\in \cW_h$ for some
positive real number $h$, and let $\zp(x)=(p_0,\dotsc,p_n)$.  Let
$i$ and $j$ be indices such that $p_i$ is either 0 or a leaner of
$x$, $p_j$ is either $h$ or a leaner of $x$, $i<j$ and $p_k$ is not
a leaner of $x$ if $i<k<j$.  Then $x_{i+1}=\cdots=x_j$, $x_i\ne
x_{i+1}$ if $i>0$, and $x_{j+1}\ne x_j$ if $j<n$.  We call
$(x_{i+1},\dotsc,x_j)$ a \emph{segment} of $x$.  In other words,
the leaners of $x$ parse $x$ into segments such that the coordinates
of $x$ in every segment are equal, and segments are maximal with
respect to this property. We call $p_i$ and $p_j$ the
\emph{endpoints} of the segment $(x_{i+1},\dotsc,x_j)$.  By the
\emph{value} of a segment we mean the value of any component of
$x$ in that segment.  By the \emph{dimension} of a segment we mean
the number of components in it.  By the \emph{height} of a segment
we mean the sum of its components.  Each endpoint of a segment is
either a leaner or an endpoint of the interval $[0,h]$.  If the left
endpoint of a segment is a left leaner, then we say that it
\emph{leans away from} the segment, and if it is a right leaner,
then we say that it \emph{leans toward} the segment.  The
situation is similar for right endpoints.

Let $y$ be a weight vector in the image of $\zF$. A weight vector
$x$ is called a \emph{minimal preimage} of $y$ if it satisfies the
following:
\begin{enumerate}
  \item[i)] $\zF(x)=y$.
  \item[ii)] If $x'$ is a weight vector with $\zF(x')=y$, then $A(x)\le
   A(x')$, with equality if and only if $x=x'$.
\end{enumerate}

\begin{lemma}\label{lemma:preimage}  Let $y$ be a weight
vector,  and let $t_1,\dotsc,t_m$ be the segments of $y$ in order.
Then $y$ is in the image of $\zF$ if and only if there is a weight
vector $x$ with segments $s_1,\dotsc,s_m$ in order which satisfies
the following:
\begin{enumerate}
  \item The height of $s_i$ equals the height of $t_i$ for
   $i\in\{1,\ldots,m\}$.
  \item For every $i\in\{1,\ldots,m\}$ the dimension of $s_i$ is the
   dimension of $t_i$ plus $\za_i-\zt_i$ where $\za_i$ is the number of
   endpoints of $t_i$ which lean away from $t_i$ and $\zt_i$ is the
   number of endpoints of $t_i$ which lean toward $t_i$.
\end{enumerate}
Furthermore, if $y$ is in the image of $\zF$ then this vector $x$ is
unique and is the minimal preimage of $y$.
\end{lemma}
\begin{proof} We first suppose that $y$ is in the image of $\zF$, and
let $w$ be a weight vector such that $\zF(w)=y$. Suppose that the
right endpoint of $t_1$ leans toward $t_1$, namely, that it is a
left leaner.  Suppose that the right endpoint of $t_1$ is partition
point $k$ of $y$.  Lemma~\ref{lemma:leaner'} implies that this left
leaner is blocked by $w$, and so it is partition point number $k-1$
of $w$.  This proves that if the right endpoint of $t_1$ leans
toward $t_1$, then the dimension of $t_1$ is at least 2. Similarly,
if the left endpoint of the last segment $t_m$ leans toward $t_m$,
then the dimension of $t_m$ is at least 2.  More generally, if both
endpoints of some segment of $y$ lean toward that segment, then the
dimension of that segment is at least 3. With the results of this
paragraph and the fact that $\sum_{i=1}^{m}(\za_i-\zt_i)=0$, we see
that conditions 1 and 2 in the statement of
Lemma~\ref{lemma:preimage} uniquely determine a weight vector $x$.

For the converse, suppose that $x$ is a weight vector with segments
$s_1,\dotsc,s_m$ in order which satisfies conditions 1 and 2 of the
lemma. An induction argument on $i$ shows that the endpoints of
$t_i$ other than 0 and $h$ are blocked by $x$ and that $x$ is
compatible with $y$. By Lemma~\ref{lemma:leaner'}, $\zF(x) = y$ and
so $y$ is in the image of $\zF$.

For the last statement, suppose that $y$ is in the image of $\zF$.
Let $x$ be the weight vector which satisfies conditions 1 and 2 of
the lemma, and let $x'$ be a weight vector with $\zF(x')=y$. Let
$i\in\{1,\ldots,m\}$. Because the endpoints of $t_i$ other than 0
and $h$ are blocked by $x'$, the number of components of $x'$
between these endpoints equals the number of components of $x$
between these endpoints, namely, the dimension of $s_i$.  But since
the components of $s_i$ are all equal, Corollary~\ref{cor:wh}
implies that the area of $s_i$ is at most the sum of the squares of
the corresponding components of $x'$. Thus $A(x)\le A(x')$, and
equality holds if and only if $x=x'$.

This proves Lemma~\ref{lemma:preimage}.
\end{proof}

This section has thus far been concerned with the definition of
$\zF$ and basic relationships between a weight vector and its image
under $\zF$.  We now turn to area estimates.  The next result,
Theorem~\ref{thm:twosegs}, prepares for Theorem~\ref{thm:estimate},
which is the key ingredient in our proof of the dumbbell theorem.

\begin{thm}\label{thm:twosegs}  Let $h$ be a positive real
number, and let $x\in \cW_h$.  Then there exists a weight vector
$w\in \cW_h$ satisfying the following conditions.
\begin{enumerate}
  \item Both $w$ and $\zV(w)$ have at most two segments with nonzero
        values. Furthermore, if they have two segments with nonzero
        values, then these segments are consecutive.
  \item The weight vector $w$ is the minimal preimage of $\zV(w)$.
  \item $A(w)\le A(x)$.
  \item $A(\zV(w))\ge A(\zV(x))$.
\end{enumerate}
\end{thm}
  \begin{proof} Set
  \begin{equation*}
W=\{z\in \cW_h:A(z)\le A(x)\text{ and }A(\zF(z))\ge A(\zF(x))\}.
  \end{equation*}
Since every element of $W$ satisfies conditions 3 and 4, it suffices
to prove that $W$ contains a weight vector satisfying conditions 1
and 2.  We will do this by first proving that $W$ contains an
element with minimal area. It follows that $w$ is the minimal
preimage of $\zF(w)$, giving condition 2. We then prove that $w$
satisfies condition 1.

In this paragraph we prove that $W$ contains an element with minimal
area.  Since $x\in W$, $W$ is not empty.  The set $\cW_h$ is compact,
and so the closure $\overline{W}$ of $W$ in $\cW_h$ is compact.  Since
the area function is continuous, there exists a weight vector $w$ in
$\overline{W}$ with minimal area.  Because the area function is
continuous and $w\in \overline{W}$, we have that $A(w)\le A(x)$.  It
remains to prove that $A(\zF(w))\ge A(\zF(x))$.  For this let $w'$ be
a weight vector in $W$ near $w$.  Lemma~\ref{lemma:reformulate} shows
that $\zF(w)$ is compatible with $w$.  Because $w'$ is near $w$, there
exists a weight vector $y$ near $\zF(w)$ which is compatible with
$w'$.  Because $y$ is near $\zF(w)$ the area of $y$ is not much larger
than the area of $\zF(w)$.  Now Lemma~\ref{lemma:reformulate} shows
that the area of $\zF(w')$ is not much larger than the area of
$\zF(w)$.  Since $A(\zF(w'))\ge A(\zF(x))$, it follows that
$A(\zF(w))\ge A(\zF(x))$.  Thus $w\in W$.  This proves that $W$
contains an element with minimal area.

We fix an element $w$ in $W$ with minimal area.
Lemma~\ref{lemma:preimage} implies that $w$ is the minimal preimage
of $\zF(w)$.  To prove Theorem~\ref{thm:twosegs} it suffices to
prove that $w$ cannot have nonconsecutive segments with nonzero
values. We do this by contradiction. So suppose that $w$ has
nonconsecutive segments with nonzero values. Let $s_1,\dotsc,s_m$ be
the segments of $w$ in order with dimensions $d_1,\dotsc,d_m$,
values $v_1,\dotsc,v_m$, and heights $h_1,\dotsc,h_m$, so that
$h_i=d_iv_i$ for every $i\in\{1,\ldots,m\}$.

In this paragraph we show that $w$ has two consecutive segments with
nonzero values. We proceed by contradiction: suppose that $w$ does
not have two consecutive segments with nonzero values. Then every
other value of $w$ is $0$. Let $i\in\{1,\dotsc,m\}$ such that
$v_i\ne 0$, $v_{i+1}=0$, and $v_{i+2}\ne 0$. Because of the symmetry
with respect to the order of the components of $w$, we may, and do,
assume that $v_i\le v_{i+2}$. Let $A_0$ be the sum of the areas of
the segments of $w$ other than $s_i$, $s_{i+1}$, and $s_{i+2}$. Then
\begin{equation*}
A(w) = A_0 + \frac{h_i^2}{d_i} + \frac{h_{i+2}^2}{d_{i+2}}.
\end{equation*}
Similarly, if $B_0$ is the sum of the areas of the segments of
$\zF(w)$ other than those corresponding to $s_i$, $s_{i+1}$, and
$s_{i+2}$, then
\begin{equation*}
A(\zF(w)) = B_0 + \frac{h_i^2}{e_i} + \frac{h_{i+2}^2}{e_{i+2}},
\end{equation*}
where
\begin{equation*}
e_i =
\left\{%
\begin{array}{ll}
    d_i+1, & \hbox{if $i=1$} \\
    d_i+2, & \hbox{if $i>1$} \\
\end{array}%
\right.
\end{equation*}
and
\begin{equation*}
e_{i+2} =
\left\{%
\begin{array}{ll}
    d_{i+2}+1, & \hbox{if $i= m$} \\
    d_{i+2}+2, & \hbox{if $i< m$.} \\
\end{array}%
\right.
\end{equation*}
We construct a new weight vector $w'$ by modifying $w$ as follows.
Where $w$ has segments $s_i$, $s_{i+1}$, $s_{i+2}$ with heights
$h_i$, $0$, $h_{i+2}$ and dimensions $d_i$, $d_{i+1}$, $d_{i+2}$,
the weight vector $w'$ has segments with heights $0$, $h_{i}$,
$h_{i+2}$ and dimensions $d_i'$, $d_i + 1$, $d_{i+2}$. The heights
and dimensions of all other segments of $w'$ equal the corresponding
heights and dimensions of $w$ except that if $i>1$, then $w'$ has
one fewer segment than $w$ because segment $s_{i-1}$ of $w$ has
value $0$ and so the components of $w$ in $s_{i+1}$ must be adjoined
to $s_{i-1}$. So
\begin{equation*}
A(w') = A_0 + \frac{h_i^2}{d_i+1} + \frac{h_{i+2}^2}{d_{i+2}} < A(w)
\le A(x).
\end{equation*}
Using Lemma~\ref{lemma:preimage} we see that
\begin{equation*}
A(\zF(w')) = B_0 + \frac{h_i^2}{d_i+1} + \frac{h_{i+2}^2}{e_{i+2}}
\ge A(\zF(w)) \ge A(\zF(x)).
\end{equation*}
This contradicts the fact that $w$ is an element in $W$ with minimal
area. Thus $w$ has two consecutive segments with nonzero values. If
$w$ only has two segments with nonzero values, then we are done.
Hence we may assume that $w$ has at least three segments with
nonzero values.

Next suppose that $w'$ is a weight vector with the same number of
segments as $w$, the segments of $w'$ have the same dimensions as the
corresponding segments of $w$ but the leaners of $w'$ are gotten by
perturbing the leaners of $w$, equivalently, the heights of the
segments of $w'$ are gotten by perturbing the heights of the segments
of $w$.  If this perturbation is small enough, then
Lemma~\ref{lemma:leaner'} implies that $\zF(w')$ is gotten from
$\zF(w)$ by exactly the same perturbation.  We will prove that there
exists such a perturbation such that $A(w')<A(w)$ and
$A(\zF(w'))=A(\zF(w))$.  Thus $w'$ is an element of $W$ with smaller
area than $w$, a contradiction which will complete the proof of
Theorem~\ref{thm:twosegs}.

To begin the construction of such a perturbation, we note that the
area of $w$ is
  \begin{equation*}
\sum_{i=1}^{m}d_iv_i^2=\sum_{i=1}^{m}d_i\left(\frac{h_i}{d_i}\right)^2
=\sum_{i=1}^{m}\frac{1}{d_i}h_i^2=\sum_{i=1}^{m}a_ih_i^2,
  \end{equation*}
where $a_i=d_i^{-1}$ for every $i\in\{1,\ldots,m\}$.  Recall that $w$
is the minimal preimage of $\zV(w)$.  As above, if $t_1,\dotsc,t_m$
are the segments of $\zF(w)$ in order with dimensions
$A_1^{-1},\dotsc,A_m^{-1}$, then the area of $\zF(w)$ is
  \begin{equation*}
\sum_{i=1}^{m}A_ih_i^2.
  \end{equation*}
This leads us to define functions $f\co \bR^m\to \bR$ and $F\co
\bR^m\to \bR$ so that
  \begin{equation*}
f(x_1,\dotsc,x_m)=\sum_{i=1}^{m}a_i(x_i+h_i)^2
  \end{equation*}
and
  \begin{equation*}
F(x_i,\dotsc,x_m)=\sum_{i=1}^{m}A_i(x_i+h_i)^2.
  \end{equation*}
To prove Theorem~\ref{thm:twosegs}, it suffices to find points
$x=(x_1,\dotsc,x_m)\in \bR^m$ arbitrarily near the origin 0 such
that $\sum_{i=1}^{m}x_i=0$, $f(x)<f(0)$, $F(x)=F(0)$, and $x_i\ge 0$
if $h_i=0$.

We do this while fixing all but three components of
$(x_1,\dotsc,x_m)$.  We specify these components later in this
paragraph.  Since only three components of $(x_1,\dotsc,x_m)$ are
allowed to vary, the functions $f$ and $F$ are really functions of
three variables.  To simplify notation, we now view $f$ and $F$ as
functions from $\bR^3$ to $\bR$ with
  \begin{equation*}
f(x,y,z)=a(x+x_0)^2+b(y+y_0)^2+c(z+z_0)^2+d
  \end{equation*}
and
  \begin{equation*}
F(x,y,z)=A(x+x_0)^2+B(y+y_0)^2+C(z+z_0)^2+D.
  \end{equation*}
Since $w$ has two consecutive segments with nonzero values, by
Lemma~\ref{lemma:preimage} $\zF(w)$ has two consecutive segments
with nonzero values. Let the variable $x$ correspond to a segment
of $\zF(w)$ with nonzero value which is adjacent to a segment with
larger value. Let $z$ correspond to a segment of $\zF(w)$ with
maximal value.  Let $y$ correspond to any segment of $\zF(w)$ with
nonzero value not already chosen.  Note that $x_0, y_0, z_0 > 0$.
Because the segment of $\zF(w)$ corresponding to $x$ has an
endpoint leaning away from it, its dimension $A^{-1}$ is at most
as large as the dimension $a^{-1}$ of the corresponding segment of
$w$. Thus $A^{-1}\le a^{-1}$, and so $A\ge a$.  Because the
segment of $\zF(w)$ corresponding to $z$ has no endpoint leaning
away from it, $C^{-1}>c^{-1}$, and so $C<c$. Because the value of
the segment of $\zF(w)$ corresponding to $z$ is larger than the
value of the segment corresponding to $x$, we have $Cz_0>Ax_0$. To
prove Theorem~\ref{thm:twosegs} it suffices to find points
$(x,y,z)\in\bR^3$ arbitrarily near $(0,0,0)$ such that $x+y+z=0$,
$f(x,y,z)<f(0,0,0)$, and $F(x,y,z)=F(0,0,0)$.

The set of all solutions $(x,y,z)$ to the equation $f(x,y,z)=f(0,0,0)$
is an ellipsoid containing the point $(0,0,0)$.  The gradient of $f$
at $(0,0,0)$ is $(2ax_0,2by_0,2cz_0)$.  Because $ax_0$, $by_0$, $cz_0$
are not all equal, the plane given by $x+y+z=0$ contains the point
$(0,0,0)$, but it is not tangent to the ellipsoid.  Thus this plane
intersects the ellipsoid in an ellipse.  It likewise intersects the
ellipsoid given by $F(x,y,z)=F(0,0,0)$ in an ellipse.  These ellipses
both lie in this plane and they have a point in common.  If they
intersect transversely at $(0,0,0)$, then it is easy to find points
arbitrarily near $(0,0,0)$ which lie within the first ellipse and lie on
the second one.  In other words, Theorem~\ref{thm:twosegs} is true if
the ellipses intersect transversely.

Finally, we assume that the ellipses are tangent at $(0,0,0)$.  Let
$L$ be the line in $\bR^3$ tangent to the ellipses at $(0,0,0)$.  Then
$L$ lies in the plane given by $x+y+z=0$ and it lies in the tangent
planes to both of the ellipsoids.  These three planes have normal
vectors given by $(1,1,1)$, $(ax_0,by_0,cz_0)$ and
$(Ax_0,By_0,Cz_0)$.  Thus these three vectors are linearly dependent.
Thus the same is true for the columns of the matrix
  \begin{equation*}
\left[\begin{matrix}1 & ax_0 & Ax_0 \\ 1 & by_0 & By_0 \\ 1 & cz_0 &
Cz_0 \end{matrix}\right].
  \end{equation*}
The cross product of any two rows is orthogonal to all three rows.
The cross product of row 1 and row 3 is
  \begin{equation*}
v=((aC-Ac)x_0z_0,Ax_0-Cz_0,cz_0-ax_0).
  \end{equation*}
Since $a\le A$ and $C<c$, we have $(aC-Ac)x_0z_0<0$.  Since
$Cz_0>Ax_0$, we have $Ax_0-Cz_0<0$.  Since $v$ is orthogonal to the
rows of the preceding matrix, it is also orthogonal to the rows of
  \begin{equation*}
\left[\begin{matrix}\frac{1}{x_0} & a & A \\ \frac{1}{y_0} & b & B \\
\frac{1}{z_0} & c & C \end{matrix}\right].
  \end{equation*}
So
  \begin{equation*}
(aC-Ac)x_0z_0\left(\frac{1}{x_0},\frac{1}{y_0},\frac{1}{z_0}\right)+
(Ax_0-Cz_0)(a,b,c)+(cz_0-ax_0)(A,B,C)=(0,0,0).
  \end{equation*}
Solving for $(a,b,c)$, we find that
  \begin{equation*}
(a,b,c)=r\left(\frac{1}{x_0},\frac{1}{y_0},\frac{1}{z_0}\right)+s(A,B,C)
  \end{equation*}
for some real numbers $r$ and $s$ with $r<0$.  Viewing $f(x,y,z)$ and
$F(x,y,z)$ as dot products of vectors except for the final additive
constants, we have that
  \begin{equation*}
\begin{aligned}
f(x,y,z)&-f(0,0,0)\\
& = (a,b,c)\cdot((x+x_0)^2,(y+y_0)^2,(z+z_0)^2)-
(a,b,c)\cdot(x_0^2,y_0^2,z_0^2)\\
&  =(a,b,c)\cdot(x^2+2x_0x,y^2+2y_0y,z^2+2z_0z)\\
&  = r\left(\frac{1}{x_0},\frac{1}{y_0},\frac{1}{z_0}\right)\cdot
   (x^2+2x_0x,y^2+2y_0y,z^2+2z_0z)\\
&  \qquad+s(A,B,C)\cdot
   (x^2+2x_0x,y^2+2y_0y,z^2+2z_0z)\\
&  = r\left[\left(\frac{1}{x_0}x^2+\frac{1}{y_0}y^2+
     \frac{1}{z_0}z^2\right)+2(x+y+z)\right]+s[F(x,y,z)-F(0,0,0)]\\
&  = r\left(\frac{1}{x_0}x^2+\frac{1}{y_0}y^2+
     \frac{1}{z_0}z^2\right)+s[F(x,y,z)-F(0,0,0)].
\end{aligned}
  \end{equation*}

So if $(x,y,z)$ is any point on the ellipse corresponding to $F$, then
$F(x,y,z)-F(0,0,0)=0$ and so $f(x,y,z)-f(0,0,0)<0$, as desired.

This proves Theorem~\ref{thm:twosegs}.

\end{proof}

\begin{thm}\label{thm:estimate}  Let $h$ be a positive real
number.  Let $x_1,\dotsc,x_{3n}$ be weight vectors in $\cW_h$ such
that $A(x_{i+1})\le A(\zV(x_i))$ for every $i\in\{1,\ldots,3n-1\}$.
Then $x_{3n}=w_h$.
\end{thm}
  \begin{proof} Lemma~\ref{lemma:reducearea} implies that $\zV$ does
not increase area, and so
  \begin{equation*}
A(x_1)\ge\cdots\ge A(x_{3n}).
  \end{equation*}
Lemma~\ref{lemma:reducearea} and Corollary~\ref{cor:wh} furthermore
imply that if two of $x_1,\dotsc,x_{3n}$ have the same area, then
$x_{3n}=w_h$.  So we may, and do, assume that $x_1,\dotsc,x_{3n}$ are
distinct.

Since $\zV(hx)=h \zV (x)$ for every weight vector $x\in \cW_1$, we may,
and do, assume that $h=1$.

Let $i\in\{1,\ldots,3n-1\}$.  Theorem~\ref{thm:twosegs} implies that
there exists a weight vector $x_i'\in \cW_1$ such that both $x'_i$ and
$\zV(x'_i)$ have at most two segments with nonzero values, the weight
vector $x'_i$ is the minimal preimage of $\zV(x'_i)$, $A(x'_i)\le
A(x_i)$ and $A(\zV(x'_i))\ge A(\zV(x_i))$.  Set $x'_{3n}=x_{3n}$.
Then
  \begin{equation*}
A(x'_{i+1})\le A(x_{i+1})\le A(\zV(x_i))\le A(\zV(x'_i))
  \end{equation*}
for every $i\in\{1,\ldots,3n-1\}$.  Thus to prove
Theorem~\ref{thm:estimate}, we may replace $x_1,\dotsc,x_{3n}$ by
$x'_1,\dotsc,x'_{3n}$.  In other words we may, and do, assume for
every $i\in\{1,\ldots,3n-1\}$ that $x_i$ and $\zV(x_i)$ do not have
nonconsecutive segments with nonzero values and that $x_i$ is the
minimal preimage of $\zV(x_i)$.

At this point the argument splits into two cases which depend on how
many of the weight vectors $x_1,\dotsc,x_{3n}$ have segments with
value 0.  In Case 1 we prove Theorem~\ref{thm:estimate} under the
assumption that at least $n$ of the weight vectors $x_1,\dotsc,x_{3n}$
have a segment with value 0.  In Case 2 we prove
Theorem~\ref{thm:estimate} under the assumption that fewer than $n$ of
the weight vectors $x_1,\dotsc,x_{3n}$ have a segment with value 0.

We begin Case 1 now.  Assume that at least $n$ of the weight vectors
$x_1,\dotsc,x_{3n}$ have a segment with value 0.  Let $z_1,\dotsc,z_n$
be distinct elements of $\{x_1,\dotsc,x_{3n}\}$ in order each of which
has a segment with value 0.  Then $A(z_{i+1})\le A(\zV(z_i))$ for every
$i\in\{1,\ldots,n-1\}$.

In this paragraph we show that in Case 1 it suffices to prove for
every $i\in\{1,\ldots,n-1\}$ that
  \begin{equation*}\linnum\label{lin:oneovern}
A(z_{i+1})(A(z_i)+1)-A(z_i)\le0.
\end{equation*}
Indeed, this inequality is equivalent to the inequality
  \begin{equation*}
A(z_{i+1})\le\frac{1}{A(z_i)^{-1}+1}.
  \end{equation*}
Since $H(z_1)=1$, we have $A(z_1)\le1$.  A straightforward induction
argument based on the preceding displayed inequality shows that
$A(z_i)\le 1/i$ for every $i\in\{1,\ldots,n\}$.  In particular
$A(z_n)\le1/n$.  But Corollary~\ref{cor:wh} implies that $w_1$ is the
unique element in $\cW_1$ with area at most $1/n$.  Hence
$z_n=x_{3n}=w_1$.  Thus to prove Theorem~\ref{thm:estimate} in Case 1
it suffices to prove line~\ref{lin:oneovern} for every
$i\in\{1,\ldots,n-1\}$.

In this paragraph we assume that there exists $i\in\{1,\ldots,n-1\}$
such that $z_i$ has only one segment with nonzero value.  Let $d$ be
the dimension of this segment.  Then $A(z_i)=1/d$.  Moreover
$\zV(z_i)$ also has just one segment with nonzero value with dimension
either $d+1$ or $d+2$.  Hence $A(\zV(z_i))$ is either $1/(d+1)$ or
$1/(d+2)$.  So line~\ref{lin:oneovern} holds.  Thus to prove
Theorem~\ref{thm:estimate} in Case 1 it suffices to prove
line~\ref{lin:oneovern} under the assumption that $z_i$ has two
segments with nonzero value.

So let $i\in\{1,\ldots,n-1\}$, and suppose that $z_i$ has two
consecutive segments $s$ and $t$ in order with nonzero value.  Let
$p$ be the right endpoint of $s$, equivalently, the left endpoint of
$t$. Let $d$ be the dimension of $s$, and let $e$ be the dimension
of $t$. The value of $s$ is $p/d$, and the value of $t$ is
$(1-p)/e$. It follows that
  \begin{equation*}\linnum\label{lin:areazi}
A(z_i)=d^{-1}p^2+e^{-1}(1-p)^2.
  \end{equation*}
Let $d'$ and $e'$ be the dimensions of the segments $s'$ and $t'$ of
$\zV(z_i)$ corresponding to $s$ and $t$.  Then
  \begin{equation*}
A(\zV(z_i))=(d')^{-1}p^2+(e')^{-1}(1-p)^2.
  \end{equation*}
By symmetry we may assume that
  \begin{equation*}
\frac{p}{d'}\le\frac{1-p}{e'}.
  \end{equation*}
The last inequality implies that $p$ is a right leaner for
$\zV(z_i)$. Hence Lemma~\ref{lemma:leaner'} implies that $p$ is a
right leaner for $z_i$ and so
  \begin{equation*}
\frac{p}{d}\le\frac{1-p}{e}.
  \end{equation*}
Combining this with line~\ref{lin:areazi} yields that
  \begin{equation*}
A(z_i)=d^{-1}p^2+e^{-1}(1-p)^2\le e^{-1}p(1-p)+e^{-1}(1-p)^2=e^{-1}(1-p).
  \end{equation*}
Similarly, we have that
  \begin{equation*}\linnum\label{lin:estvi}
A(\zV(z_i))\le(e')^{-1}(1-p).
  \end{equation*}

Because $z_i$ has a segment with value 0 and $z_i$ is the minimal
preimage of $\zV(z_i)$, it is also true that $\zV(z_i)$ has a
segment with value 0.  These facts and Lemma~\ref{lemma:preimage}
imply that if $\zV(z_i)$ has a segment with value 0 preceding $s'$,
then $d' = d$ and $e'\ge e+1$.  If $\zV(z_i)$ does not have a
segment with value 0 preceding $s'$, then $d'=d-1$ and $e'=e+2$.

Suppose that $\zV(z_i)$ has a segment with value 0 preceding $s'$.
The previous paragraph shows that $d' = d$ and $e'\ge e+1$.
Combining this with line~\ref{lin:estvi} yields that
$A(\zV(z_i))\le(e+1)^{-1}(1-p)$.  Hence
  \begin{equation*}
\begin{aligned}
A(z_{i+1})&(A(z_i)+1)-A(z_i)\\
   & \le A(\zV(z_i))(A(z_i)+1)-A(z_i)\\
  & \le A(\zV(z_i))A(z_i)+A(\zV(z_i))-A(z_i)\\
  & \le\frac{1}{e(e+1)}(1-p)^2+\frac{1}{d'}p^2+
    \frac{1}{e'}(1-p)^2-\frac{1}{d}p^2-\frac{1}{e}(1-p)^2\\
  & \le\frac{1}{e(e+1)}(1-p)^2+\frac{1}{d}p^2+
    \frac{1}{e+1}(1-p)^2-\frac{1}{d}p^2-\frac{1}{e}(1-p)^2\\
  & \le \frac{1}{e(e+1)}(1-p)^2-\frac{1}{e(e+1)}(1-p)^2\\
  & \le 0.
\end{aligned}
  \end{equation*}
Thus line~\ref{lin:oneovern} is true if $\zV(z_i)$ has a segment
with value 0 preceding $s'$.

Next assume that $\zV(z_i)$ does not have a segment with value 0
preceding $s'$.  Then the next-to-last paragraph shows that $d'=
d-1$ and $e'= e+2$.  Combining this with line~\ref{lin:estvi} yields
that $A(\zV(z_i))\le (e+2)^{-1}(1-p)$.  We also have that
  \begin{equation*}
\frac{p}{d-1}=\frac{p}{d'}\le \frac{1-p}{e'}=\frac{1-p}{e+2}.
  \end{equation*}
Hence
  \begin{equation*}
\begin{aligned}
A(z_{i+1})&(A(z_i)+1)-A(z_i)\\
  & \le A(\zV(z_i))(A(z_i)+1)-A(z_i)\\
  & \le A(\zV(z_i))A(z_i)+A(\zV(z_i))-A(z_i)\\
  & \le\frac{1}{e(e+2)}(1-p)^2+\frac{1}{d'}p^2+
    \frac{1}{e'}(1-p)^2-\frac{1}{d}p^2-\frac{1}{e}(1-p)^2\\
  & \le\frac{1}{e(e+2)}(1-p)^2+\frac{1}{d-1}p^2+
    \frac{1}{e+2}(1-p)^2-\frac{1}{d}p^2-\frac{1}{e}(1-p)^2\\
  & \le\frac{1}{e(e+2)}(1-p)^2+\frac{1}{(d-1)d}p^2-
    \frac{2}{e(e+2)}(1-p)^2\\
  & \le\frac{1}{e(e+2)}(1-p)^2+\frac{1}{e(e+2)}(1-p)^2-
    \frac{2}{e(e+2)}(1-p)^2\\
  & \le 0.
\end{aligned}
  \end{equation*}
Thus line~\ref{lin:oneovern} is true if $\zV(z_i)$ does not have a
segment with value 0 preceding $s'$.

The proof of Theorem~\ref{thm:estimate} is now complete in Case 1,
namely, Theorem~\ref{thm:estimate} is proved if at least $n$ of the
weight vectors $x_1,\dotsc,x_{3n}$ have a segment with value 0.

Now we proceed to Case 2.  Suppose that fewer than $n$ of the weight
vectors $x_1,\dotsc,x_{3n}$ have a segment with value 0.  Then at
least $2n$ of the weight vectors $x_1,\dotsc,x_{3n}$ do not have a
segment with value 0.  Let $z_1,\dotsc,z_{2n}$ be distinct elements of
$\{x_1,\dotsc,x_{3n}\}$ in order none of which has a segment with
value 0.  Then $A(z_{i+1})\le A(\zV(z_i))$ for every
$i\in\{1,\ldots,2n-1\}$.  As we have seen, to finish the proof of
Theorem~\ref{thm:estimate}, it suffices to prove that $z_{2n}= w_1$.
We do this by contradiction.  Suppose that $z_{2n}\ne w_1$.  As we
have seen, it follows that $z_i\ne w_1$ for $i\in\{1,\ldots,2n\}$.

Let $i\in\{1,\ldots,2n-1\}$.  Then $z_i$ has at most two segments,
$z_i$ has no segment with value 0, and $z_i$ is the minimal preimage
of $\zV(z_i)$.  The only weight vector in $\cW_1$ with one segment
is $w_1$, so $z_i$ has two segments neither of which has value 0.

Let the first segment of $z_i$ have right endpoint $p$ and dimension
$d$.  Then the second segment of $z_i$ has left endpoint $p$ and
dimension $n-d$.  So
  \begin{equation*}
\begin{aligned}
A(z_i) & = \frac{1}{d}p^2+\frac{1}{(n-d)}(1-p)^2  =
\frac{(n-d)p^2+d(1-p)^2}{d(n-d)}\\
 & = \frac{np^2-dp^2+d-2dp+dp^2}{d(n-d)}  =
\frac{n^2p^2-2dnp+dn}{d(n-d)n}\\
 & = \frac{(np-d)^2-d^2+dn}{d(n-d)n}  =
\frac{(d-np)^2}{d(n-d)n}+\frac{1}{n}.
\end{aligned}
  \end{equation*}
Set
  \begin{equation*}
B(z_i)=\frac{(d-np)^2}{d(n-d)}.
  \end{equation*}
Then
  \begin{equation*}
A(z_i)=\frac{1}{n}B(z_i)+\frac{1}{n}.
  \end{equation*}
Since $z_i$ has two segments neither of which has value 0, the same
is true of $\zV(z_i)$.  We define $B(\zV(z_i))$ in the same way that
we define $B(z_i)$.  Then the given inequality $A(z_{i+1})\le
A(\zV(z_i))$ is equivalent to the inequality $B(z_{i+1})\le
B(\zV(z_i))$.  We focus on this latter inequality.

Let $a$ be the positive real number such that $B(z_i)=a^2$.  Then
  \begin{equation*}
\begin{gathered}
\frac{(d-np)^2}{d(n-d)}  = a^2 \\
(d-np)^2  =a^2d(n-d) \\
d^2-nd+a^{-2}(d-np)^2  = 0\\
(d-\frac{n}{2})^2+a^{-2}(d-np)^2  = \frac{n^2}{4}\\
\left(\frac{2d}{n}-1\right)^2+a^{-2}\left(\frac{2d}{n}-2p\right)^2  = 1.
\end{gathered}
\end{equation*}

We are led to consider the family $\cE$ of all ellipses of the form
$x^2+a^{-2}y^2=1$, where $a$ is a positive real number.  The open
upper halves of the ellipses in $\cE$ fill the open half infinite
strip $S=\{(x,y)\in \bR^2:|x|<1,y>0\}$.

The weight vector $z_i$ determines the point
  \begin{equation*}
P(z_i)=(x_i,y_i)=\left(\frac{2d}{n}-1,\frac{2d}{n}-2p\right)
  \end{equation*}
in $\bR^2$, and we next show that we may assume that $P(z_i)\in S$.
It is easy to see that $|x_i|<1$, so what is needed is the inequality
$y_i>0$.  Since $z_i$ and $\zV(z_i)$ have exactly one leaner, by
symmetry we may, and do, assume that this leaner is a right leaner.
Hence
  \begin{equation*}
\begin{gathered}
\frac{1}{d}p<\frac{1}{n-d}(1-p) \\
p(n-d)<d(1-p)\\
pn<d.
\end{gathered}
  \end{equation*}
It follows that $y_i>0$, and so $P(z_i)\in S$.  It is even true that
$y_i<x_i+1$, and so $P(z_i)$ lies in the triangle $T=\{(x,y)\in
\bR^2:|x|<1,0<y<x+1\}$.

The point of $\bR^2$ corresponding to $\zV(z_i)$ likewise lies in
$T$.  Moreover since the first segment of $z_i$ has right endpoint $p$
and dimension $d$, it follows that the first segment of $\zV(z_i)$ has
right endpoint $p$ and dimension $d-1$.  In other words, the point of
$T$ corresponding to $\zV(z_i)$ is
  \begin{equation*}
\left(\frac{2(d-1)}{n}-1,\frac{2(d-1)}{n}-2p\right)=
\left(x_i-\frac{2}{n},y_i-\frac{2}{n}\right).
  \end{equation*}

Now we can reformulate our problem as follows.  Let $\zt\co \bR^2\to
\bR^2$ be the translation defined by
$\zt(x,y)=\left(x-\frac{2}{n},y-\frac{2}{n}\right)$.  We begin with a
point $t_1\in T$.  Let $E_1$ be the ellipse in $\cE$ containing
$t_1$.  We have that $\zt(t_1)\in T$ and $\zt(t_1)$ lies within
$E_1$.  Let $E'_1$ be the ellipse in $\cE$ containing $\zt(t_1)$.  Let
$t_2$ be a point of $T$ either on or within $E'_1$.  We iterate this
construction to obtain points $t_1,\dotsc,t_{2n}\in T$.  To complete
the proof of Theorem~\ref{thm:estimate} it suffices to prove that it
is impossible to construct points $t_1,\dotsc,t_{2n}$ in this way.

To this end, let $A$ be the subset of $\bR^2$ which is the union of
the closed line segment joining $(0,0)$ with $(0,1)$ and the closed
line segment joining $(0,1)$ with $(1,2)$.  See
Figure~\ref{fig:alpha}.  We define a function $\za\co T\to \bR$ as
follows.  Let $t\in T$.  There exists a unique ellipse $E\in \cE$
which contains $t$, and $A\cap E$ consists of a single point.  Let
$\za(t)$ be the $y$-coordinate of this point.

  \begin{figure}[h!]
\centerline{\includegraphics{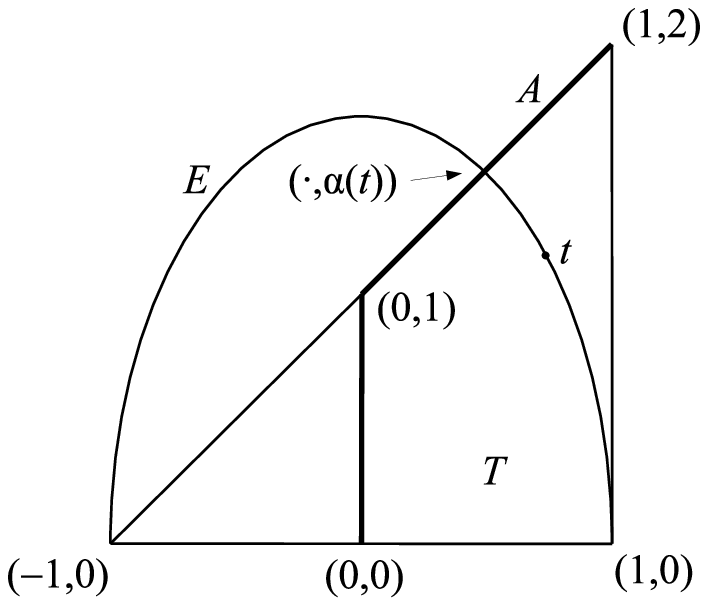}}
\caption{ Defining $\za(t)$.}
\label{fig:alpha}
  \end{figure}

We will prove that if $t_1,\dotsc,t_{2n}$ are points in $T$ as in the
next-to-last paragraph, then $\za(t_{i+1})<\za(t_i)-1/n$ for every
$i\in\{1,\ldots,2n-1\}$.  This suffices to complete the proof of
Theorem~\ref{thm:estimate}.  Because $t_{i+1}$ lies either on or within
the ellipse in $\cE$ containing $\zt(t_i)$, it actually suffices to
prove that $\za(\zt(t_i))<\za(t_i)-1/n$ for every
$i\in\{1,\ldots,2n-1\}$, which is what we do.

So let $i\in\{1,\ldots,2n-1\}$.  Let $E_i$ and $E'_i$ be the
ellipses in $\cE$ which contain $t_i$ and $\zt(t_i)$.

First suppose that the $x$-coordinates of both $t_i$ and $\zt(t_i)$
are nonnegative.  See Figure~\ref{fig:case1}.  In this case it is
clear that we actually have that $\za(\zt(t_i))<\za(t_i)-2/n$, as
desired.

  \begin{figure}[h!]
\centerline{\includegraphics{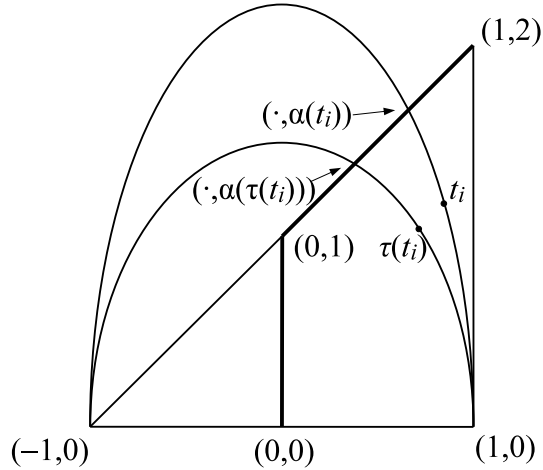}}
\caption{When the $x$-coordinates of both $t_i$ and $\zt(t_i)$
are nonnegative.}
\label{fig:case1}
  \end{figure}

Next suppose that the $x$-coordinates of both $t_i$ and $\zt(t_i)$ are
nonpositive.  See Figure~\ref{fig:case2}.
  \begin{figure}[h!]
\centerline{\includegraphics{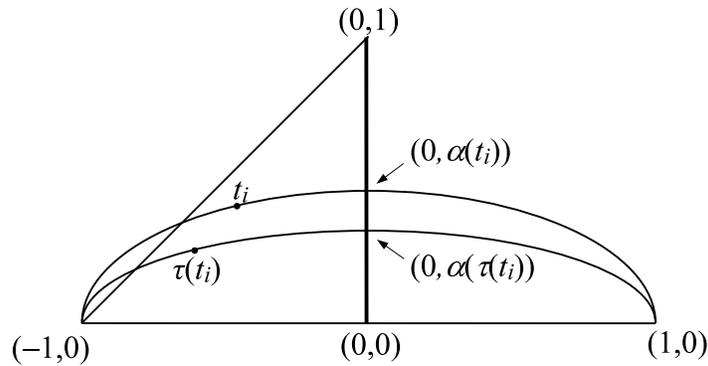}}
\caption{When the $x$-coordinates of both $t_i$ and $\zt(t_i)$ are
nonpositive.}
\label{fig:case2}
  \end{figure}
Since $x^2+a^{-2}y^2=1$, we have that $y^2=a^2(1-x^2)$.
Differentiating this with respect to $x$ yields that $2yy'=-2a^2x$.
Solving for $y'$ and replacing $a^2$ with $y^2/(1-x^2)$ yields that

  \begin{equation*}\linnum\label{lin:yprime}
y'=\frac{xy}{x^2-1}.
  \end{equation*}
Since $y<x+1$ and $x/(x^2-1)\ge 0$ if $-1< x\le 0$, we have that
$y'<x/(x-1)$.  It follows that $y'<1/2$ throughout $T$.  This implies
that the vertical distance from $t_i$ to $E'_i$ is at least $1/n$.
Line~\ref{lin:yprime} implies that $E_i$ increases more than $E'_i$ as
they approach $A$ from the left.  It follows that
$\za(\zt(t_i))<\za(t_i)-1/n$, as desired.

Finally, if the $x$-coordinate of $t_i$ is positive and the
$x$-coordinate of $\zt(t_i)$ is negative, then an argument similar to
that in the previous paragraph shows again that
$\za(\zt(t_i))<\za(t_i)-1/n$.

This completes the proof of Theorem~\ref{thm:estimate}.

\end{proof}

\section{The proof of the dumbbell theorem }
\label{sec:proof}\nosubsections

We prove the dumbbell theorem in this section.

Let $D$ be a dumbbell.  Let $\zr$ be a fat flow optimal weight
function for $D$.  Let $h=H_\zr$.  We scale $\zr$ so that it is a
sum of minimal skinny cuts, making $h$ an integer.  Suppose that $D$
has bar height $n$. Let $C_1,\dotsc,C_{3n}$ be the leftmost $3n$
columns in $B$ in left-to-right order, and let $D_1,\dotsc,D_{3n}$
be the rightmost $3n$ columns in $B$ in right-to-left order.

For every $i\in\{1,\ldots,3n\}$, applying $\zr$ to the tiles in $C_i$
obtains a weight vector $y_i\in \bR^n$ and applying $\zr$ to the tiles
in $D_i$ obtains a weight vector $z_i\in \bR^n$.  From the assumption
that $H_\zr=h$, it follows that $H(y_1)\ge h$, and so by scaling $y_1$
we can obtain a weight vector $x_1\in \cW_h$ such that $A(x_1)\le
A(y_1)$.  Let $j$ be the largest element of $\{1,\dotsc,3n\}$ such
that there exist weight vectors $x_1,\dotsc,x_j\in \cW_h$ with
$A(x_1)\le A(y_1)$ and if $j>1$, then $A(x_j)\ge A(y_j)$ and
$A(x_{i+1})\le A(\zV(x_i))$ for every $i\in\{1,\ldots,j-1\}$.
Similarly, let $k$ be the largest element of $\{1,\dotsc,3n\}$ such
that there exist weight vectors $x_1,\dotsc,x_k\in \cW_h$ with
$A(x_1)\le A(z_1)$ and if $k>1$, then $A(x_k)\ge A(z_k)$ and
$A(x_{i+1})\le A(\zV(x_i))$ for every $i\in\{1,\ldots,k-1\}$.

In this paragraph we define a new weight function $\zs$ on $D$.  Let
$t$ be a tile of $D$.  If $t$ is contained in either one of the
balls of $D$ or one of the columns
$C_1,\dotsc,C_{j-1},D_{k-1},\dotsc,D_1$, then $\zs(t)=\zr(t)$.  If
$t$ is in $B$ and strictly between $C_{3n}$ and $D_{3n}$, then
$\zs(t)=h/n$.  Now suppose that $t$ is contained in $C_j$.  We have
that $\zr$ is a sum of $h$ minimal skinny cuts.  We define $\zs(t)$
to be the number of these skinny cuts $c$ with the property that as
$c$ proceeds from the left side of $D$ to $C_j$, the tile $t$ is the
first tile of $C_j$ in $c$.  This defines $\zs(t)$ for every tile
$t$ in $C_j$.  These values determine a weight vector
$\overline{x}_j\in \cW_h$ in the straightforward way.  If $j<3n$,
then we use the weight vector $x_{j+1}=\zV(\overline{x}_j)$ in the
straightforward way to define $\zs(t)$ for every tile $t$ in
$C_{j+1}$.  We inductively set $x_{i+1}=\zV(x_i)$ and use this
weight vector in the straightforward way to define $\zs(t)$ for
every tile $t$ in $C_{i+1}$ for every $i\in\{j+1,\ldots,3n-1\}$.
This defines $\zs(t)$ for every tile $t$ in $C_j,\dotsc,C_{3n}$.  We
define $\zs(t)$ for every tile $t$ in $D_k,\dotsc,D_{3n}$
analogously.  The definition of $\zs$ is now complete.

By assumption there exist weight vectors $x_1,\dotsc,x_j\in \cW_h$
with $A(x_1)\le A(y_1)$ and if $j>1$, then $A(x_j)\ge A(y_j)$ and
$A(x_{i+1})\le A(\zV(x_i))$ for every $i\in\{1,\ldots,j-1\}$.  In
defining $\zs$ we constructed weight vectors
$\overline{x}_j,x_{j+1},\dotsc,x_{3n}\in \cW_h$ with
$x_{j+1}=\zV(\overline{x}_j)$ and $x_{i+1}=\zV(x_i)$ for every
$i\in\{j+1,\ldots,3n-1\}$.  We redefine $x_j$ to be $\overline{x}_j$.
Then $x_1,\dotsc,x_{3n}$ are weight vectors in $\cW_h$ such that
$A(x_{i+1})\le A(\zV(x_i))$ for every $i\in\{1,\ldots,3n-1\}$.
Theorem~\ref{thm:estimate} implies that $x_{3n}=w_h$.  The situation
is analogous at the right end of $B$.

In this paragraph we prove that $H_\zs=h$.  Because $\zs$ is constant
on the union of the columns of $B$ between $C_{3n-1}$ and $D_{3n-1}$,
the restriction of $\zs$ to these columns is a sum of skinny cuts.
The definition of $\zV$ implies that these skinny cuts can be extended
to the columns of $B$ from $C_j$ to $D_k$ so that the restriction of
$\zs$ to the columns from $C_j$ to $D_k$ is a sum of skinny cuts.  The
definition of $\zs$ implies that these skinny cuts can be extended to
all of $D$ so that $\zs$ is greater than or equal to a weight
function which is a sum of $h$ skinny cuts.  We conclude that
$H_\zs\ge h$.  But since the columns of $D$ from $C_j$ to $D_k$ have
$\zs$-height $h$, it is in fact the case that $H_\zs=h$.

In this paragraph we prove that $A_\zs = A_\zr$.  We proceed by
contradiction.  Suppose that $A_\zs\ne A_\zr$.  Since $H_\zs=h=H_\zr$
and $\zr$ is optimal, it follows that $A_\zr<A_\zs$.  By definition
$\zr$ and $\zs$ agree away from the columns of $B$ from $C_j$ to
$D_k$.  Hence the values of $\zr$ on the tiles of one of the columns
$C$ from $C_j$ to $D_k$ determines a weight vector in $\cW_{h'}$ for
some $h'\ge h$ whose area is less than the weight vector in $\cW_h$
gotten from the values of $\zs$ on the tiles of $C$.  If $C$ is
between $C_{3n}$ and $D_{3n}$, then the weight vector determined by
$\zs$ is $w_h$.  Corollary~\ref{cor:wh} shows that this is impossible
because the area of $w_h$ is minimal.  Thus $C$ is either one of the
columns $C_j,\dotsc,C_{3n}$ or $D_k,\dotsc,D_{3n}$.  By symmetry we
may assume that $C=C_i$ for some $i\in\{j,\ldots,3n\}$.  Then
$A(y_i)<A(x_i)$.  By construction we have that $A(\overline{x}_j)\le
A(y_j)$.  Since we replaced the original value of $x_j$ by
$\overline{x}_j$, we have that $A(x_j)\le A(y_j)$, and so $i>j$.  Now
we have a contradiction to the choice of $j$.  This proves that
$A_\zs=A_\zr$.

Now we have that $H_\zs=H_\zr$ and $A_\zs=A_\zr$.  It follows that
$\zs=\zr$.  Since $\zs$ is virtually bar uniform, we see that $\zr$ is
virtually bar uniform.

This proves the dumbbell theorem.

\section{Notes }\label{sec:notes}\nosubsections

In this section we state without proof a number of results related
to the dumbbell theorem.

We say that a dumbbell $D$ is vertically convex if whenever $x$ and
$y$ are points in $D$ such that the line segment joining them is
vertical, then this line segment lies in $D$.  We say that the sides
of $D$ are extreme if the sides of $D$ are contained in vertical
lines such that $D$ lies between these lines.  Suppose that the
vertical lines determined by our tiling of the plane intersect the
$x$-axis in exactly the set of integers.  We say that a skinny cut
for $D$ is monotonic if it has an underlying path which is the graph
of a function and that the restriction of this function to the
interval determined by any pair of consecutive integers is either
monotonically increasing or decreasing.  If the dumbbell $D$ is
vertically convex and its sides are extreme, then every optimal
weight function for $D$ is a sum of monotonic minimal skinny cuts.

We recall Lemma~\ref{lemma:exists}.  Let $R$ be a rectangle tiled by
$n$ rows and $m$ columns of squares $T_{ij}$.  Let
$x=(x_{11},\dotsc,x_{n1})$ be a weight vector in $\bR^n$.  Then there exists
a unique weight function $\zr$ for $R$ with minimal area subject to
the conditions that $\zr$ is a sum of strictly monotonic skinny cuts
and $\zr(T_{i1})=x_i$ for every $i\in\{1,\ldots,n\}$.  The skinny cut
function $\zV$ is defined so that if $m=2$, then
$\zV(x)=(\zr(T_{12}),\dotsc,\zr(T_{n2}))$.  It is in fact true for every
$m\ge2$ that $\zV^{j-1}(x)=(\zr(T_{1j}),\dotsc,\zr(T_{nj}))$ for every
$j\in\{2,\ldots,m\}$.

Let $h$ be a positive real number, and let $x\in \cW_h$.  Let
$\zp(x)=(p_0,\dotsc,p_n)$.  We define a nonnegative integer $\zm_i$
for every $i\in\{1,\ldots,n-1\}$.  Let $i\in\{1,\ldots,n-1\}$.  If
$p_i<\frac{hi}{n}$, then $\zm_i$ is the number of partition points
$p_i,\dotsc,p_{n-1}$ in the half-closed half-open interval
$[p_i,\frac{hi}{n})$.  If $p_i>\frac{hi}{n}$, then $\zm_i$ is the
number of partition points $p_1,\dotsc,p_i$ in the half-open
half-closed interval $(\frac{hi}{n},p_i]$.  If $p_i=\frac{hi}{n}$,
then $\zm_i=0$.  Set $\zm=\max\{\zm_1,\dotsc,\zm_{n-1}\}$.  Note
that $\zm\le n-1$, and that $\zm=n-1$ for the weight vector
$(0,0,0,\dotsc,h)$.  It turns out that $\zV^m(x)=w_h$ for some
nonnegative integer $m$ if and only if $m\ge \zm$.  This implies
that $\zV^{n-1}(x)=w_h$ for every weight vector $x\in \cW_h$.  In
other words, if $x_1,\dotsc,x_{n-1}$ are weight vectors in $\cW_h$
such that $x_{i+1}=\zV(x_i)$ for every $i\in\{1,\ldots,n-2\}$, then
$x_{n-1}=w_h$.  Compare this with Theorem~\ref{thm:estimate}.  We
see that the integer $3n$ in Theorem~\ref{thm:estimate} cannot be
replaced by an integer less than $n-1$.  It is in fact true that
there exists a positive real number $\ze$ such that the integer $3n$
in Theorem~\ref{thm:estimate} cannot be replaced by an integer less
than $(1+\ze)n$.  It is probably true that $3n$ can be replaced by
$3n/2$.

Not surprisingly, the skinny cut function $\zV$ is continuous.  What
might be surprising is that it is in fact piecewise affine.  If $c_n$
denotes the number of affine pieces of $\zV$, then the sequence
$(c_n)$ satisfies a quadratic linear recursion with eigenvalue
$1+\sqrt{2}$.

There exists a theory of combinatorial moduli ``with boundary
conditions'', as alluded to by Lemma~\ref{lemma:exists}.  For
simplicity we deal with a rectangle $R$ with $n$ rows of tiles.  Let
$h$ be a positive real number, and let $x$ be a weight vector in
$\bR^n$ with height $h$.  Here are two ways to define optimal weight
functions with boundary conditions.  In the first way we minimize
area over all weight functions on $R$ with fat flow height $h$ such
that the values of $\zr$ on the first column of $R$ give $x$.  In
the second way we maximize the fat flow modulus over all weight
functions on $R$ whose values on the first column of $R$ give a
scalar multiple of $x$.  It turns out that these constructions are
equivalent.  They yield a weight function which is unique up to
scaling.  Interpreted properly, most of the results of \cite{Magnus}
hold in this setting. For example, the optimal weight function is a
sum of minimal fat flows (modulo the first column).  A
straightforward modification of the algorithm for computing optimal
weight functions even applies in this setting.

\end{document}